\documentclass[a4paper,11pt,reqno]{amsart}
\usepackage{graphicx}
\usepackage{multirow}
\usepackage{amsmath,amssymb,amsfonts}
\usepackage{amsthm}
\usepackage[title]{appendix}
\usepackage{xcolor}
\usepackage{textcomp}
\usepackage{manyfoot}
\usepackage{algorithm}
\usepackage{algorithmicx}
\usepackage{algpseudocode}
\usepackage{listings}
\usepackage{mathtools}
\usepackage{enumitem}
\usepackage{bbm}
\usepackage{tikz}
\usepackage[backend=bibtex, style=numeric, doi=false, url=false, isbn=false]{biblatex}
\usepackage{booktabs}
\usepackage[english]{babel}
\usetikzlibrary{backgrounds}
\usetikzlibrary{arrows}

\pgfdeclarelayer{edgelayer}
\pgfdeclarelayer{nodelayer}
\pgfsetlayers{background,edgelayer,nodelayer,main}

\tikzstyle{None}=[inner sep=0mm]
\tikzstyle{Empty node}=[fill=black, draw=black, shape=circle, inner sep=0cm, minimum height=0.2cm]
\tikzstyle{Bead}=[fill=white, draw=black, shape=circle, minimum height=0.4cm, inner sep=0cm]
\tikzstyle{Move it}=[->]
\tikzstyle{Extra box}=[-, fill=none, dashed]
\tikzstyle{Border edge}=[-, thick, fill=none]
\tikzstyle{Grey diagram}=[-, fill={rgb,255: red,122; green,122; blue,122}, thick]

\theoremstyle{plain}
\newtheorem{theorem}{Theorem}[section]
\newtheorem{lemma}[theorem]{Lemma}

\newtheorem{corollary}[theorem]{Corollary}

\theoremstyle{definition}

\newtheorem{example}[theorem]{Example}

\theoremstyle{remark}
\newtheorem{remark}[theorem]{Remark}

\makeatletter
\@namedef{subjclassname@2020}{%
	\textup{2020} Mathematics Subject Classification}
\makeatother

\newcommand{\Z}{\mathbb{Z}}
\newcommand{\Q}{\mathbb{Q}}

\DeclareMathOperator{\sgn}{sgn}
\DeclareMathOperator{\Abc}{Abc}
\DeclareMathOperator{\Par}{Par}
\DeclareMathOperator{\Com}{Com}
\DeclareMathOperator{\supp}{supp}
\DeclareMathOperator{\wt}{wt}
\DeclareMathOperator{\sh}{sh}

\newcommand{\List}[2]{#1_1,#1_2,\dots ,#1_{#2}}

\title[Proof of the plethystic Murnaghan--Nakayama rule]{Proof of the plethystic Murnaghan--Nakayama rule using Loehr's labelled abacus}
\author{Pavel Turek}
\date{August 13, 2025}
\address{Department of Mathematics, Royal Holloway, University of London, Egham, Surrey TW20 0EX, UK}
\email{pkah149@live.rhul.ac.uk}
\begin{document}	 
	\begin{abstract}
		The plethystic Murnaghan--Nakayama rule describes how to decompose the product of a Schur function and a plethysm of the form $p_r\circ h_m$ as a sum of Schur functions. We provide a short, entirely combinatorial proof of this rule using the labelled abaci introduced in \fullcite{LoehrAbacus10}.
	\end{abstract}
	\maketitle 
	
	\thispagestyle{empty}
	
	\section{Introduction}\label{Sec intro}
	
	In 2010, Loehr \cite{LoehrAbacus10} introduced a labelled abacus as a combinatorial model for antisymmetric polynomials $a_{\beta}$. By considering appropriate moves of labelled beads and their collisions, he proved standard formulas for decompositions of products of Schur polynomials with other symmetric polynomials, namely Pieri's rule, Young's rule, the Murnaghan--Nakayama rule and the Littlewood--Richardson rule, as well as the equivalence of the combinatorial and the algebraic definitions of Schur polynomials and a formula for inverse Kostka numbers. We follow the slogan `when beads bump, objects cancel' from \cite{LoehrAbacus10} and enrich this collection of results by proving the plethystic Murnaghan--Nakayama rule.
	
	\begin{theorem}[Plethystic Murnaghan--Nakayama rule]\label{Theorem pMN rule}
		Let $\mu$ be a partition and $r$ and $m$ be positive integers. Then
		\begin{equation*}
		s_{\mu}(p_r\circ h_m)=\sum_{\mu\subseteq\lambda\in\Par(|\mu|+rm)}\sgn_r(\lambda/\mu)s_{\lambda}. 
		\end{equation*}
	\end{theorem}
	See \S\ref{Sec part} for definitions of $\sgn_r$ and $r$-decomposable skew partitions, which are the skew partitions for which $\sgn_r$ is non-vanishing.
	
	More precisely, we prove the formula in Theorem~\ref{Theorem pMN rule} for symmetric polynomials in $N$ variables where $N\geq |\mu|+rm$. This is equivalent to Theorem~\ref{Theorem pMN rule} as both sides of the formula have degree $|\mu|+rm$. One can easily extend the result to a decomposition of $s_{\mu}\left(p_{\rho}\circ h_{\nu} \right) $ as a sum of Schur functions for any non-empty partitions $\rho$ and $\nu$ by iterating Theorem~\ref{Theorem pMN rule} and using the formula $p_{\rho}\circ h_{\nu}=\prod_{i,j} p_{\rho_i}\circ h_{\nu_j}$.
	
	The plethystic Murnaghan--Nakayama rule is a generalisation of the usual Murnaghan--Nakayama rule, which can be obtained from Theorem~\ref{Theorem pMN rule} by letting $m=1$, that is by replacing the plethysm $p_r\circ h_m$ with $p_r$. By letting $r=1$ instead, one obtains Young's rule which describes the decomposition of $s_{\mu}h_m$ as a sum of Schur functions.
	
	While a description of the plethysm $p_r\circ h_m = p_r\circ s_{(m)}$ is known and follows from Theorem~\ref{Theorem pMN rule} after letting $\mu=\o$, in general, it is a difficult problem to decompose a plethysm as a sum of Schur functions; see \cite[Problem~9]{StanleyPositivity00}, which asks for a decomposition of plethysms of the form $s_{(a)}\circ s_{(b)}$. Plethysms play an important role not only in the study of symmetric functions but also in the representation theory of symmetric groups and general linear groups; see \cite[Chapter~7:~Appendix~2]{StanleyEnumerativeII99}. The connection of plethysms and representation theory was used, for instance, in \cite{deBoeckPagetWildonPlethysms21} to find the maximal constituents of plethysms of Schur functions using the highest weight vectors.
	
	The plethystic Murnaghan--Nakayama rule appeared first in \cite[p.29]{DesarmenienLeclercThibonHallFunctions94}, where it was proved using Muir's rule. Since then, it has been proved using several different methods: in \cite[Proposition~4.3]{EvseevPagetWildonDeflation14} characters of symmetric groups are used, \cite{WildonPMNRule16} uses James' (unlabelled) abacus and induction on $m$ and \cite[Corollary~3.8]{cao2022plethystic} uses vertex operators. In comparison to these proofs, our elementary proof using the labelled abaci arises naturally by `merging' the proofs from \cite{LoehrAbacus10} of the Murnaghan--Nakayama rule and Young's rule.
	
	Since the publication of the original paper introducing the labelled abaci, Loehr has used it to prove the Cauchy product identities in \cite{LoehrAbacusCauchyProducts19}, and together with Wills they introduced abacus-tournaments to study Hall--Littlewood polynomials in \cite{LoehrWillsAbacusTournament16}.
	
	\section{Definitions}\label{Sec defns}
	
	\subsection{Partitions}\label{Sec part}
	
	A \textit{partition} $\lambda=(\lambda_1,\lambda_2,\dots,\lambda_l)$ is a non-increasing sequence of positive integers. The \textit{size} of a partition $\lambda$, denoted by $|\lambda|$, equals $\sum_{i=1}^l \lambda_i$. We call the number of elements of $\lambda$ the \textit{length} of $\lambda$ and denote it by $\ell(\lambda)$. We use the convention that for $i>\ell(\lambda)$ we have $\lambda_i=0$, and we allow ourselves to attach extra zeros to a partition without changing it. We write $\Par_{\leq N}$ for the set of partitions of length at most $N$, $\Par(n)$ for the set of partitions of size $n$ and $\Par_{\leq N}(n)$ for the intersection of these two sets.
	
	The \textit{Young diagram} of a partition $\lambda$ is $Y_{\lambda}=\left\lbrace (i,j)\in\mathbb{N}^2 : i\leq \ell(\lambda), j\leq \lambda_i \right\rbrace $ and we refer to its elements as \textit{boxes}. We write $\mu\subseteq\lambda$ whenever $Y_{\mu}\subseteq Y_{\lambda}$. A \textit{skew partition} $\lambda/\mu$ is a pair of partitions $\mu\subseteq\lambda$ and its Young diagram is $Y_{\lambda/\mu}=Y_{\lambda}\setminus Y_{\mu}$. We define the \textit{top} of a skew partition $\lambda/\mu$, denoted as $t(\lambda/\mu)$, to be $0$ if $\lambda=\mu$, and the least $i$ such that $\lambda_i\neq \mu_i$ otherwise. We similarly define the \textit{bottom} $b(\lambda/\mu)$ of a skew partition by replacing the word `least' with `greatest'.
	
	Let $r$ be a positive integer. An \textit{$r$-border strip} is a skew partition $\lambda/\mu$ consisting of $r$ edge-adjacent boxes such that for all $(i,j)\in Y_{\lambda/\mu}$ we have $(i+1,j+1)\notin Y_{\lambda/\mu}$. It follows from the definition that for any partition $\lambda$ and a non-negative integer $t$, there is at most one $r$-border strip $\lambda/\mu$ with $t(\lambda/\mu)=t$. A skew partition $\lambda/\mu$ is \textit{$r$-decomposable} if there are partitions
	\begin{equation}\label{Eq chain}
	\mu=\gamma^{(0)}\subseteq\gamma^{(1)}\subseteq\dots\subseteq\gamma^{(d)}=\lambda
	\end{equation}
	such that the skew partition $\gamma^{(i+1)}/\gamma^{(i)}$ is an $r$-border strip for all $0\leq i\leq d-1$ and $t(\gamma^{(1)}/\gamma^{(0)})\geq t(\gamma^{(2)}/\gamma^{(1)})\geq \dots\geq t(\gamma^{(d)}/\gamma^{(d-1)})$. If such a decomposition exists, it is unique as there is a unique choice for $\gamma^{(d-1)}$ as $\lambda/\gamma^{(d-1)}$ is an $r$-border strip with $t(\lambda/\gamma^{(d-1)})=t(\lambda/\mu)$ and an inductive argument then applies. Examples of Young diagrams of $r$-border strips and an $r$-decomposable skew partition are in Figure~\ref{Figure border strips}.
	
	\begin{figure}[ht]
		\begin{tikzpicture}[x=0.5cm, y=0.5cm]
		\begin{pgfonlayer}{nodelayer}
		\node [style=None] (0) at (-8, 0) {};
		\node [style=None] (1) at (-6.5, 0) {};
		\node [style=None] (2) at (-6.5, 1.5) {};
		\node [style=None] (3) at (-0.5, 1.5) {};
		\node [style=None] (4) at (-8, 9) {};
		\node [style=None] (5) at (-5, 1.5) {};
		\node [style=None] (6) at (-5, 4.5) {};
		\node [style=None] (7) at (-3.5, 4.5) {};
		\node [style=None] (8) at (-3.5, 7.5) {};
		\node [style=None] (9) at (-0.5, 7.5) {};
		\node [style=None] (10) at (-0.5, 9) {};
		\node [style=None] (11) at (-2, 7.5) {};
		\node [style=None] (12) at (-2, 3) {};
		\node [style=None] (13) at (-3.5, 3) {};
		\node [style=None] (14) at (-3.5, 1.5) {};
		\node [style=None] (15) at (-0.5, 4.5) {};
		\node [style=None] (16) at (1, 4.5) {};
		\node [style=None] (17) at (1, 7.5) {};
		\node [style=None] (18) at (4, 7.5) {};
		\node [style=None] (19) at (4, 9) {};
		\node [style=None] (20) at (-6.5, 9) {};
		\node [style=None] (21) at (-5, 9) {};
		\node [style=None] (22) at (-3.5, 9) {};
		\node [style=None] (23) at (-2, 9) {};
		\node [style=None] (24) at (-8, 7.5) {};
		\node [style=None] (25) at (-3.5, 6) {};
		\node [style=None] (26) at (-8, 6) {};
		\node [style=None] (27) at (-8, 4.5) {};
		\node [style=None] (28) at (-8, 3) {};
		\node [style=None] (29) at (-8, 1.5) {};
		\node [style=None] (30) at (2.5, 9) {};
		\node [style=None] (31) at (1, 9) {};
		\node [style=None] (32) at (2.5, 7.5) {};
		\node [style=None] (33) at (1, 6) {};
		\node [style=None] (34) at (-0.5, 3) {};
		\node [style=None] (35) at (-2, 1.5) {};
		\node [style=None] (36) at (-4.25, 2.25) {};
		\node [style=None] (37) at (-4.25, 3.75) {};
		\node [style=None] (38) at (-2.75, 3.75) {};
		\node [style=None] (39) at (-2.75, 6.75) {};
		\node [style=None] (40) at (-2.75, 2.25) {};
		\node [style=None] (41) at (-1.25, 2.25) {};
		\node [style=None] (42) at (-1.25, 6.75) {};
		\node [style=None] (43) at (0.25, 5.25) {};
		\node [style=None] (44) at (0.25, 8.25) {};
		\node [style=None] (45) at (3.25, 8.25) {};
		\node [style=None] (46) at (-4, 2.25) {$1$};
		\node [style=None] (47) at (-1, 2.25) {$2$};
		\node [style=None] (48) at (0.5, 5.25) {$3$};
		\end{pgfonlayer}
		\begin{pgfonlayer}{edgelayer}
		\draw [style=Grey diagram] (1.center)
		to (0.center)
		to (4.center)
		to (10.center)
		to (9.center)
		to (8.center)
		to (7.center)
		to (6.center)
		to (5.center)
		to (2.center)
		to cycle;
		\draw [style=Border edge] (8.center) to (11.center);
		\draw [style=Border edge] (11.center) to (12.center);
		\draw [style=Border edge] (12.center) to (13.center);
		\draw [style=Border edge] (13.center) to (14.center);
		\draw [style=Border edge] (14.center) to (5.center);
		\draw [style=Border edge] (5.center) to (6.center);
		\draw [style=Border edge] (6.center) to (7.center);
		\draw [style=Border edge] (7.center) to (8.center);
		\draw [style=Border edge] (11.center) to (9.center);
		\draw [style=Border edge] (9.center) to (3.center);
		\draw [style=Border edge] (3.center) to (14.center);
		\draw [style=Border edge] (15.center) to (10.center);
		\draw [style=Border edge] (10.center) to (19.center);
		\draw [style=Border edge] (19.center) to (18.center);
		\draw [style=Border edge] (18.center) to (17.center);
		\draw [style=Border edge] (17.center) to (16.center);
		\draw [style=Border edge] (16.center) to (15.center);
		\draw (24.center) to (17.center);
		\draw (26.center) to (33.center);
		\draw (27.center) to (15.center);
		\draw (28.center) to (34.center);
		\draw (29.center) to (2.center);
		\draw (20.center) to (2.center);
		\draw (21.center) to (6.center);
		\draw (22.center) to (13.center);
		\draw (23.center) to (35.center);
		\draw (10.center) to (9.center);
		\draw (31.center) to (17.center);
		\draw (30.center) to (32.center);
		\draw [style=Extra box] (36.center) to (37.center);
		\draw [style=Extra box] (37.center) to (38.center);
		\draw [style=Extra box] (38.center) to (39.center);
		\draw [style=Extra box] (40.center) to (41.center);
		\draw [style=Extra box] (41.center) to (42.center);
		\draw [style=Extra box] (43.center) to (44.center);
		\draw [style=Extra box] (44.center) to (45.center);
		\end{pgfonlayer}
		\end{tikzpicture}
		\caption{Let $\gamma^{(0)}=\mu=(5,3,3,2,2,1), \gamma^{(1)}=(5,4,4,4,3,1), \gamma^{(2)}=(5,5,5,5,5,1)$ and $\gamma^{(3)}=\lambda=(8,6,6,5,5,1)$. The above dashed lines, labelled by $i=1,2,3$, pass through the Young diagrams of the $5$-border strips $\gamma^{(i)}/\gamma^{(i-1)}$. The bottoms of these $5$-border strips are $5,5$ and $3$, respectively, while their tops are $2,2$ and $1$, respectively. As the tops are in non-increasing order, the skew partition $\lambda/\mu$ is $5$-decomposable.}
		\label{Figure border strips}
	\end{figure}
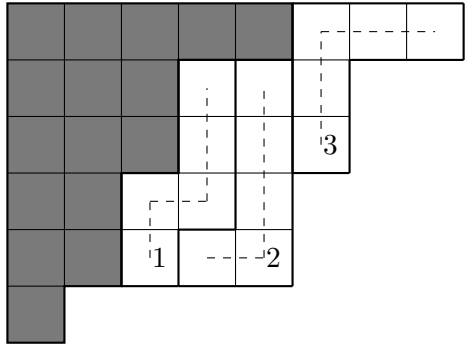
	
	For an $r$-border strip $\lambda/\mu$ we define its $\textit{sign}$ denoted by $\sgn(\lambda/\mu)$ as $(-1)^{b(\lambda/\mu)-t(\lambda/\mu)}$. For any skew partition $\lambda/\mu$ we then let $\sgn_r(\lambda/\mu)=\sgn(\gamma^{(1)}/\gamma^{(0)})\sgn(\gamma^{(2)}/\gamma^{(1)})\cdots\sgn(\gamma^{(d)}/\gamma^{(d-1)})$ where $\gamma^{(i)}$ are as in (\ref{Eq chain}) if $\lambda/\mu$ is $r$-decomposable, and $\sgn_r(\lambda/\mu)=0$ otherwise. Looking at Figure~\ref{Figure border strips}, the signs of the $5$-border strips there are $-1,-1$ and $1$, respectively, and hence $\sgn_5(\lambda/\mu)=1$.
	
	\subsection{Symmetric polynomials}\label{Sec sym fns}
	
	A \textit{composition} of a non-negative integer $m$ is a sequence $\beta=(\beta_1,\beta_2,\dots,\beta_N)$ of non-negative integers such that $\sum_{i=1}^N \beta_i=m$. The \textit{length} of a composition is the number of its elements and we write $\Com_N(m)$ for the set of compositions of $m$ of length $N$.
	
	We now introduce the required elements of the ring of symmetric polynomials in $N$ variables, called $\Lambda_N$, as defined, for instance, in \cite[\S I.2]{MacdonaldPolynomials95}. For a positive integer $m$, the \textit{complete homogeneous symmetric polynomial} $h_{m}(\List{x}{N})$ is defined as $\sum_{\beta\in \Com_N(m)}x^{\beta}$, where $x^{\beta}$ is the monomial $x_1^{\beta_1}x_2^{\beta_2}\dots x_N^{\beta_N}\in\Z\left[x_1,x_2,\dots,x_N \right] $. If $r$ is a positive integer, the \textit{power sum symmetric polynomial} $p_r(\List{x}{N})$ is defined as $\sum_{i=1}^N x_i^r$.
	
	If $g$ is an element of $\Lambda_N$, we define the \textit{plethysm} $p_r\circ g (\List{x}{N})$ as $g(\List{x^r}{N})$. In particular, if $g=h_m$, we get $p_r\circ h_m(\List{x}{N})=\sum_{\beta\in \Com_N(m)}x^{r\beta}$, where $r\beta=(\List{r\beta}{N})$. One can define a plethysm $f\circ g$ for any elements $f$ and $g$ of $\Lambda_N$ by extending the map $\cdot\circ g$ to an endomorphism of the $\Q$-algebra $\Q\otimes_{\Z}\Lambda_N$. It can be checked that for any $g\in\Lambda_N$ we have $p_r\circ g = g \circ p_r$; thus, in particular, $p_r\circ h_m = h_m \circ p_r$.
	
	To define the final ingredient, Schur polynomials, we introduce the antisymmetric polynomials $a_{\beta}$: for a positive integer $N$ and a composition $\beta$ of length $N$ we let $a_{\beta}=\det(x_i^{\beta_j})_{i,j\leq N}$. Given a partition $\lambda$ of length at most $N$, we now define the \textit{Schur polynomial}
	\begin{equation}\label{Eq Schur definition}
	s_{\lambda}(\List{x}{N})=\frac{a_{\lambda+\delta(N)}}{a_{\delta(N)}},
	\end{equation}
	where $\delta(N)=(N-1,N-2,\dots,0)$ and $\lambda+\delta(N)=(\lambda_1+N-1,\lambda_2+N-2,\dots, \lambda_N)$. While compared to other definitions such as \cite[Definition~7.10.1]{StanleyEnumerativeII99}, it is not immediately obvious that Schur polynomials are polynomials, we use this definition as one requires antisymmetric polynomials to use the labelled abaci. 
	
	\begin{example}\label{Example symmetric polynomials}
		Let $N=3$. Then
		\begin{align*}
		h_{2}(x_1,x_2,x_3)&=x_1^2+x_2^2+x_3^2+x_1x_2+x_2x_3+x_1x_3,\\
		p_4(x_1,x_2,x_3)&=x_1^4 + x_2^4 +x_3^4,\\
		p_4\circ h_2(x_1,x_2,x_3) &=x_1^8+x_2^8+x_3^8+x_1^4x_2^4+x_2^4x_3^4+x_1^4x_3^4,\\
		s_{(2,1)}(x_1,x_2,x_3)&=\frac{a_{(4,2,0)}}{a_{(2,1,0)}}=\frac{(x_1^2-x_2^2)(x_2^2-x_3^3)(x_1^2-x_3^2)}{(x_1-x_2)(x_2-x_3)(x_1-x_3)}\\
		&=(x_1+x_2)(x_2+x_3)(x_1+x_3)\\
		&=x_1^2x_2 + x_1^2x_3+x_1x_2^2+x_1x_3^2+x_2^2x_3+x_2x_3^2+2x_1x_2x_3.\\
		\end{align*}
	\end{example}
	
	\subsection{Labelled abacus}\label{Sec abacus}
	
	Most of our terminology and notation for labelled abaci comes from \cite{LoehrAbacus10}. A \textit{labelled abacus with $N$ beads} is a sequence $w=(w_0,w_1,w_2,\dots)$ indexed from $0$ with precisely $N$ non-zero entries, which are $1,2,\dots,N$. For $1\leq B\leq N$, we let $w^{-1}(B)$ be the index $i$ such that $w_i=B$. We write $\iota_1(w)>\iota_2(w)>\dots>\iota_N(w)$ for the indices $i$ such that $w_i$ is non-zero and define the \textit{support} $\supp(w)$ to be $\left\lbrace \iota_i(w) : 1\leq i\leq N \right\rbrace $. Finally, the \textit{sign} $\sgn(w)$ is the sign of the permutation $\sigma_w\in S_N$ given by $\sigma_w(B)=w_{\iota_{B}(w)}$.
	
	\begin{example}\label{Example sign}
		If $w=(5,0,6,4,1,0,0,3,0,2,0,0,\dots)$, a labelled abacus with $6$ beads, then $\sigma_w=(1\; 2\; 3)(5\;6)$. Hence $\sgn(w)=-1$.
	\end{example}
	
	One should imagine that a labelled abacus $w$ consists of a single runner with positions $0,1,2,\dots$, where the position $i$ is empty if $w_i=0$, and is occupied by a bead labelled by $w_i$ otherwise. The value $w^{-1}(B)$ is the position of bead $B$, the support is the set of the non-empty positions and $\iota_t(w)$ is the $t$-th largest occupied position. The permutation of beads in $w$, starting from beads ordered in decreasing order, is then $\sigma_w$. With this in mind, we introduce the following intuitive terminology.
	
	Fix positive integer $r$ and $B\neq C\leq N$ and write $y=w^{-1}(B)$ and $z=w^{-1}(C)$ for the positions of beads $B$ and $C$, respectively. A labelled abacus $w'$ is obtained from $w$ by \textit{swapping beads $B$ and $C$} if $w'_{y}=C$, $w'_{z}=B$ and $w'_i=w_i$ otherwise. Bead $B$ is \textit{$r$-mobile} if $w_{y+r}=0$. If that is the case, a labelled abacus $w'$ is obtained from $w$ by \textit{$r$-moving bead $B$} if $w'_y=0$, $w'_{y+r}=B$ and $w'_i=w_i$ otherwise. If bead $B$ in not $r$-mobile, we say that it \textit{$r$-collides} with bead $w_{y+r}$. Similarly, bead $B$ is \textit{left-$r$-mobile} if $y\geq r$ and $w_{y-r}=0$. If that is the case, a labelled abacus $w'$ is obtained from $w$ by \textit{$r$-moving bead $B$ leftwards} if $w$ is obtained from $w'$ by $r$-moving bead $B$. Finally, for $t\leq N$, the \textit{$t$-th rightmost bead} of $w$ is $\sigma_w(t)$.
	
	It is easy to see how the sign changes when performing the above operations. To state the formula, we define the \textit{number of beads} between positions $i_1<i_2$ as $|\supp(w)\cap \left\lbrace   i_1+1,i_1+2,\dots,i_2-1\right\rbrace   |$.
	
	\begin{lemma}\label{Lemma sign change}
		Let $w$ be a labelled abacus with $N$ beads. Fix $B\leq N$ and write $y=w^{-1}(B)$ for the position of bead $B$.
		\begin{enumerate}[label=\textnormal{(\roman*)}]
			\item If $C\leq N$ and $C\neq B$ and $w'$ is obtained from $w$ by swapping beads $B$ and $C$, then $\sgn(w')=-\sgn(w)$.
			\item If bead $B$ is $r$-mobile for some chosen positive integer $r$ and $w'$ is obtained from $w$ by $r$-moving bead $B$, then $\sgn(w')=(-1)^u\sgn(w)$, where $u$ is the number of beads between $y$ and $y+r$. 
		\end{enumerate}
	\end{lemma}
	
	\begin{proof}
		In (i), $\sigma_{w'}\sigma_w^{-1}$ is the transposition $(B\; C)$. In (ii), $\sigma_{w'}\sigma_w^{-1}$ is a $\left( u+1\right) $-cycle.
	\end{proof}
	
	For a labelled abacus $w$ with $N$ beads we define its \textit{weight} $\wt(w)$ as the monomial $\prod_{i\in\supp(w)}x_{w_i}^i$. We also define the \textit{shape} $\sh(w)$ to be the partition $(\iota_1(w)-N+1, \iota_2(w)-N+2,\dots, \iota_N(w))$, and given $\lambda\in\Par_{\leq N}$ we write $\Abc_N(\lambda)$ for the set of labelled abaci with $N$ beads and shape $\lambda$. Thus $\Abc_N(\lambda)$ contains $N!$ elements. An example of labelled abaci is in Figure~\ref{Figure labelled abacus}.
	
	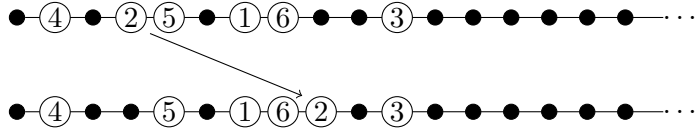
\begin{figure}[ht]
		\begin{tikzpicture}[x=0.5cm, y=0.5cm]
		\begin{pgfonlayer}{nodelayer}
		\node [style=Empty node] (0) at (0.5, 0) {};
		\node [style=Bead] (1) at (1.5, 0) {$4$};
		\node [style=Bead] (2) at (3.5, 0) {$2$};
		\node [style=Bead] (3) at (4.5, 0) {$5$};
		\node [style=Bead] (4) at (6.5, 0) {$1$};
		\node [style=Bead] (5) at (7.5, 0) {$6$};
		\node [style=Bead] (6) at (10.5, 0) {$3$};
		\node [style=Empty node] (7) at (2.5, 0) {};
		\node [style=Empty node] (8) at (8.5, 0) {};
		\node [style=Empty node] (9) at (5.5, 0) {};
		\node [style=Empty node] (10) at (9.5, 0) {};
		\node [style=Empty node] (11) at (11.5, 0) {};
		\node [style=Empty node] (12) at (12.5, 0) {};
		\node [style=None] (13) at (17.5, 0) {};
		\node [style=None] (14) at (18, 0) {$\dots$};
		\node [style=Empty node] (15) at (0.5, -2.5) {};
		\node [style=Bead] (16) at (1.5, -2.5) {$4$};
		\node [style=Bead] (17) at (8.5, -2.5) {$2$};
		\node [style=Bead] (18) at (4.5, -2.5) {$5$};
		\node [style=Bead] (19) at (6.5, -2.5) {$1$};
		\node [style=Bead] (20) at (7.5, -2.5) {$6$};
		\node [style=Bead] (21) at (10.5, -2.5) {$3$};
		\node [style=Empty node] (22) at (2.5, -2.5) {};
		\node [style=Empty node] (23) at (3.5, -2.5) {};
		\node [style=Empty node] (24) at (5.5, -2.5) {};
		\node [style=Empty node] (25) at (9.5, -2.5) {};
		\node [style=Empty node] (26) at (11.5, -2.5) {};
		\node [style=Empty node] (27) at (12.5, -2.5) {};
		\node [style=None] (28) at (17.5, -2.5) {};
		\node [style=None] (29) at (18, -2.5) {$\dots$};
		\node [style=Empty node] (30) at (13.5, 0) {};
		\node [style=Empty node] (31) at (14.5, 0) {};
		\node [style=Empty node] (32) at (13.5, -2.5) {};
		\node [style=Empty node] (33) at (14.5, -2.5) {};
		\node [style=None] (34) at (4, -0.425) {};
		\node [style=None] (35) at (8, -2.075) {};
		\node [style=Empty node] (36) at (15.5, 0) {};
		\node [style=Empty node] (37) at (16.5, 0) {};
		\node [style=Empty node] (38) at (15.5, -2.5) {};
		\node [style=Empty node] (39) at (16.5, -2.5) {};
		\end{pgfonlayer}
		\begin{pgfonlayer}{edgelayer}
		\draw (0) to (13.center);
		\draw (15) to (28.center);
		\draw [style=Move it] (34.center) to (35.center);
		\end{pgfonlayer}
		\end{tikzpicture}
		\caption{The upper labelled abacus with $6$ beads $w$ has support $\left\lbrace 10,7,6,4,3,1 \right\rbrace $. The permutation $\sigma_w$ equals $(1\; 3)(2\; 6\; 4\; 5)$ and thus $\sgn(w)=1$. We have, for instance, $w^{-1}(5)=4$ and $w^{-1}(1)=6$. The weight of $w$ is $x_1^6x_2^3x_3^{10}x_4x_5^{4}x_6^7$ and the shape of $w$ is $(5,3,3,2,2,1)$. Bead $4$ is not $5$-mobile, but the other beads are. By $5$-moving bead $2$ we obtain the lower labelled abacus $w'$. One computes that $\sigma_{w'}\sigma_w^{-1}=(2\; 5\; 1 \; 6)$, which is in accordance with the proof of Lemma~\ref{Lemma sign change}(ii).}
		\label{Figure labelled abacus}
	\end{figure}
	
	The importance of labelled abaci comes from the simple identity
	\begin{equation}\label{Eq antisymmetric}
	a_{\lambda+\delta(N)}=\sum_{w\in\Abc_N(\lambda)}\sgn(w)\wt(w),
	\end{equation}
	which holds for any $\lambda\in\Par_{\leq N}$; see \cite[p.1359]{LoehrAbacus10}. Note that the identity is just the expansion of $a_{\lambda+\delta(N)}=\det(x_i^{\lambda_j+N-j})_{i,j\leq N}$.
	
	\section{Proof of the plethystic Murnaghan--Nakayama rule}\label{Sec proof}
	
	The following is an immediate consequence of a well-known result connecting moves on an (unlabelled) abacus and removals of border strips. 
	
	\begin{lemma}\label{Lemma r-border strips}
		Let $r,t$ and $N$ be positive integers such that $t\leq N$. For $\lambda\in\Par_{\leq N}$ and $w\in\Abc_N(\lambda)$ the following holds:
		\begin{enumerate}[label=\textnormal{(\roman*)}]
			\item There is a bijection $\theta$ between left-$r$-mobile beads of $w$ and $r$-border strips of the form $\lambda/\mu$ given by mapping bead $B$ to an $r$-border strip $\lambda/\mu$, where $\mu$ is the shape of the labelled abacus obtained from $w$ by $r$-moving bead $B$ leftwards.
			\item If $\lambda/\mu$ is an $r$-border strip with top $t$, then $\theta^{-1}(\lambda/\mu)$ is the $t$-th rightmost bead of $w$.
			\item With $\mu$ as in (ii), the number of beads between positions $\iota_t(w)$ and $\iota_t(w)-r$ equals $b(\lambda/\mu)-t(\lambda/\mu)$.
			\item Continuing with the same $\mu$, there is a bijection $\phi:\Abc_N(\lambda)\to\Abc_N(\mu)$ given by $r$-moving the $t$-th rightmost bead leftwards. 
		\end{enumerate} 
	\end{lemma}
	
	\begin{proof}
		Part (i) (without labels) is \cite[Lemma~2.7.13]{JamesKerberSymmetric81}. Now suppose that $w'$ is obtained from $w$ by $r$-moving bead $B$ leftwards. If $j$ is the least index in which $\sh(w)$ and $\sh(w')$ differ, then bead $B$ is the $j$-th rightmost bead of $w$. Similarly, if $j$ is the largest such index, then bead $B$ is the $j$-th rightmost bead of $w'$. Thus we deduce (ii) and (iii). Finally, (iv) follows from (i) and (ii).
	\end{proof}
	
	\begin{example}\label{Example moving a bead}
		Let $w$ be the lower labelled abacus from Figure~\ref{Figure labelled abacus} and $\lambda=\sh(w)=(5,4,4,4,3,1)$. Since the second rightmost bead of $w$ is left-$5$-mobile, there is a corresponding $5$-border strip $\lambda/\mu$ with top $2$. Indeed, this is the $5$-border strip labelled by $1$ from Figure~\ref{Figure border strips}. The bijection from Lemma~\ref{Lemma r-border strips}(iv) then pairs the labelled abaci in Figure~\ref{Figure labelled abacus}.
	\end{example}
	
	The next step is to iterate the previous lemma. To do this, for positive integers $r, m, N$ and $\lambda,\mu\in\Par_{\leq N}$, we define a set $K_N^{r,m}(\mu,\lambda)$ as the set of sequences $(w^{(0)},w^{(1)},\dots,w^{(m)})$ of labelled abaci with $N$ beads such that $\sh(w^{(0)})=\mu$, $\sh(w^{(m)})=\lambda$, for all $1\leq j\leq m$ the labelled abacus $w^{(j)}$ is obtained from $w^{(j-1)}$ by $r$-moving a bead, say bead $w^{(j-1)}_{i_j}$, and the inequalities $i_1<i_2<\dots<i_m$ hold. We refer the reader to Figure~\ref{Figure process} for a diagrammatic example of two such sequences.
	
	\begin{lemma}\label{Lemma r-decomposable}
		Let $r, m$ and $N$ be positive integers and $\lambda,\mu\in\Par_{\leq N}$.
		\begin{enumerate}[label=\textnormal{(\roman*)}]
			\item The set $K_N^{r,m}(\mu,\lambda)$ is empty unless $\lambda/\mu$ is an $r$-decomposable skew partition of size $rm$.
			\item If $\lambda/\mu$ is an $r$-decomposable skew partition of size $rm$, then the map $(w^{(0)},w^{(1)},\dots,w^{(m)})\mapsto w^{(m)}$ is a bijection from $K_N^{r,m}(\mu,\lambda)$ to $\Abc_N(\lambda)$.
			\item For $(w^{(0)},w^{(1)},\dots,w^{(m)})\in K_N^{r,m}(\mu,\lambda)$ we have that $\sgn(w^{(m)}) = \sgn_r(\lambda/\mu)\sgn(w^{(0)})$.
		\end{enumerate}
	\end{lemma}
	
	\begin{proof}
		For any sequence of labelled abaci $(w^{(0)},w^{(1)},\dots,w^{(m)})$ let $\gamma^{(j)}=\sh(w^{(j)})$. From Lemma~\ref{Lemma r-border strips}(i), the condition that $w^{(j)}$ is obtained from $w^{(j-1)}$ by $r$-moving a bead, say bead $w^{(j-1)}_{i_j}$, implies that $\gamma^{(j)}/\gamma^{(j-1)}$ is an $r$-border strip. Suppose that this is the case for all $1\leq j\leq m$. Let bead $w^{(j-1)}_{i_j}$ be the $t_j$-th rightmost bead of $w^{(j)}$. The key observation is that $i_1<i_2<\dots<i_m$ if and only if $t_1\geq t_2\geq\dots\geq t_m$, which, by Lemma~\ref{Lemma r-border strips}(ii), is equivalent to $t(\gamma^{(1)}/\gamma^{(0)})\geq t(\gamma^{(2)}/\gamma^{(1)})\geq\dots\geq t(\gamma^{(m)}/\gamma^{(m-1)})$.
		
		Hence if $(w^{(0)},w^{(1)},\dots,w^{(m)})$ lies in $K_N^{r,m}(\mu,\lambda)$, then $\mu=\gamma^{(0)}\subseteq\gamma^{(1)}\subseteq\dots\subseteq \gamma^{(m)}=\lambda$ is a chain witnessing that $\lambda/\mu$ is $r$-decomposable, as in (\ref{Eq chain}); thus (i) is proven. Moreover, since the chain (\ref{Eq chain}) is unique, there is a unique choice of shapes of the labelled abaci in any sequence $(w^{(0)},w^{(1)},\dots,w^{(m)})\in K_N^{r,m}(\mu,\lambda)$. We can now apply Lemma~\ref{Lemma r-border strips}(iv) $m$-times to obtain (ii). Finally, Lemma~\ref{Lemma sign change}(ii) and Lemma~\ref{Lemma r-border strips}(iii) show that if $(w^{(0)},w^{(1)},\dots,w^{(m)})\in K_N^{r,m}(\mu,\lambda)$, then $\sgn(w^{(j)})= \sgn(\gamma^{(j)}/\gamma^{(j-1)}) \sgn(w^{(j-1)})$ for all $1\leq j\leq m$. Multiplying these equalities, we obtain (iii).
	\end{proof}
	
	We can rephrase this result to obtain a characterisation of $r$-decomposable partitions. In the statement, one should bear in mind that in (ii) and (iii) the $r$-moves are made consecutively, and thus a bead may $r$-move multiple times. 
	
	\begin{corollary}\label{Cor r-decomposable}
		Let $r$ and $N$ be positive integers and let $\lambda,\mu\in\Par_{\leq N}$. The following are equivalent:
		\begin{enumerate}[label=\textnormal{(\roman*)}]
			\item $\lambda/\mu$ is an $r$-decomposable skew partition.
			\item There are labelled abaci $w\in\Abc_N(\mu)$ and $w'\in\Abc_N(\lambda)$ such that $w'$ is obtained from $w$ by a series of $r$-moves of beads from positions $\List{i}{m}$, where $i_1<i_2<\dots<i_m$.
			\item For each $w'\in\Abc_N(\lambda)$ there exists $w\in\Abc_N(\mu)$ such that $w'$ is obtained from $w$ by a series of $r$-moves of beads from positions $\List{i}{m}$, where $i_1<i_2<\dots<i_m$.
		\end{enumerate}
		Moreover, if (i)--(iii) hold true, then the choice of $w$ and the series of $r$-moves in (iii) is unique, $|\lambda|=|\mu|+mr$ where $m$ is the number of $r$ moves in (ii) and (iii) and $\sgn(w')=\sgn_r(\lambda/\mu)\sgn(w)$.
	\end{corollary}
	
	We are now ready to prove the main theorem.
	
	\begin{proof}[Proof of Theorem~\ref{Theorem pMN rule}]
		Fix a partition $\mu$ and positive integers $r,m$ and $N$ such that $N\geq|\mu|+rm$. Using the definition of Schur polynomials by antisymmetric polynomials in (\ref{Eq Schur definition}), our desired equality (in $N$ variables) becomes
		\begin{equation}\label{Eq anti pMN rule}
		a_{\mu+\delta(N)}\left( p_r\circ h_m(\List{x}{N})\right) =\sum_{\mu\subseteq\lambda\in\Par(|\mu|+rm)}\sgn_r(\lambda/\mu) a_{\lambda+\delta(N)}.
		\end{equation}
		Using (\ref{Eq antisymmetric}) and the definition of the plethysm $p_r\circ h_m$, we can expand the left-hand side as
		\begin{equation}\label{Eq LHS expansion}
		\sum_{\substack{w\in\Abc_N(\mu) \\ \beta\in\Com_{N}(m)}}\sgn(w)\wt(w)x^{r \beta}=\sum_{\substack{w\in\Abc_N(\mu) \\ \beta\in\Com_{N}(m)}}\sgn(w,\beta)\wt_r(w,\beta),
		\end{equation}
		where the \textit{weight} $\wt_r(w,\beta)$ equals $\wt(w)x^{r \beta}$ and the \textit{sign} $\sgn(w,\beta)$ is just $\sgn(w)$.
		
		Given a labelled abacus $w\in\Abc_N(\mu)$ and a composition $\beta\in\Com_{N}(m)$, we consider a process on $w$ in which we read $w$ from left and every time we see bead $B$ with $\beta_B\geq 1$ we attempt to $r$-move it, provided that we have not already $r$-moved it $\beta_B$-times.
		
		In more detail, we use $v$ and $\alpha$ to denote the current labelled abacus and composition, respectively, during the process. At the start, we set $v=w$ and $\alpha=\beta$. For $i=0,1,\dots$ we look at $B=v_i$. If it is zero, we move to the next $i$. Otherwise, we look at $\alpha_B$. If it is zero, we move to the next $i$. Otherwise, we check whether bead $B$ is $r$-mobile. If it is not, we terminate the process and say that the pair $(w,\beta)$ is \textit{unsuccessful}. If it is $r$-mobile, we update $\alpha$ by decreasing $\alpha_{B}$ by $1$ and also update $v$ by $r$-moving bead $B$. After the updates, if $\alpha$ is the zero sequence we terminate the process and say that the pair $(w,\beta)$ is \textit{successful}. Otherwise, we move to the next $i$, working, of course, with the updated $v$ and $\alpha$ (thus the next bead we attempt to $r$-move may be the same bead, though, it does not have to). See Figure~\ref{Figure process} for an example.
		
		\begin{figure}[ht]
			\begin{tikzpicture}[x=0.5cm, y=0.5cm]
			\begin{pgfonlayer}{nodelayer}
			\node [style=Empty node] (0) at (0.5, 0) {};
			\node [style=Bead] (1) at (1.5, 0) {$4$};
			\node [style=Bead] (2) at (3.5, 0) {$2$};
			\node [style=Bead] (3) at (4.5, 0) {$5$};
			\node [style=Bead] (4) at (6.5, 0) {$1$};
			\node [style=Bead] (5) at (7.5, 0) {$6$};
			\node [style=Bead] (6) at (10.5, 0) {$3$};
			\node [style=Empty node] (7) at (2.5, 0) {};
			\node [style=Empty node] (8) at (8.5, 0) {};
			\node [style=Empty node] (9) at (5.5, 0) {};
			\node [style=Empty node] (10) at (9.5, 0) {};
			\node [style=Empty node] (11) at (11.5, 0) {};
			\node [style=Empty node] (12) at (12.5, 0) {};
			\node [style=None] (13) at (17.5, 0) {};
			\node [style=None] (14) at (18, 0) {$\dots$};
			\node [style=Empty node] (15) at (0.5, -2.5) {};
			\node [style=Bead] (16) at (1.5, -2.5) {$4$};
			\node [style=Bead] (17) at (8.5, -2.5) {$2$};
			\node [style=Bead] (18) at (4.5, -2.5) {$5$};
			\node [style=Bead] (19) at (6.5, -2.5) {$1$};
			\node [style=Bead] (20) at (7.5, -2.5) {$6$};
			\node [style=Bead] (21) at (10.5, -2.5) {$3$};
			\node [style=Empty node] (22) at (2.5, -2.5) {};
			\node [style=Empty node] (23) at (3.5, -2.5) {};
			\node [style=Empty node] (24) at (5.5, -2.5) {};
			\node [style=Empty node] (25) at (9.5, -2.5) {};
			\node [style=Empty node] (26) at (11.5, -2.5) {};
			\node [style=Empty node] (27) at (12.5, -2.5) {};
			\node [style=None] (28) at (17.5, -2.5) {};
			\node [style=None] (29) at (18, -2.5) {$\dots$};
			\node [style=Empty node] (30) at (13.5, 0) {};
			\node [style=Empty node] (31) at (14.5, 0) {};
			\node [style=Empty node] (32) at (13.5, -2.5) {};
			\node [style=Empty node] (33) at (14.5, -2.5) {};
			\node [style=Empty node] (36) at (0.5, -5) {};
			\node [style=Bead] (37) at (1.5, -5) {$4$};
			\node [style=Bead] (38) at (8.5, -5) {$2$};
			\node [style=Bead] (39) at (9.5, -5) {$5$};
			\node [style=Bead] (40) at (6.5, -5) {$1$};
			\node [style=Bead] (41) at (7.5, -5) {$6$};
			\node [style=Bead] (42) at (10.5, -5) {$3$};
			\node [style=Empty node] (43) at (2.5, -5) {};
			\node [style=Empty node] (44) at (3.5, -5) {};
			\node [style=Empty node] (45) at (5.5, -5) {};
			\node [style=Empty node] (46) at (4.5, -5) {};
			\node [style=Empty node] (47) at (11.5, -5) {};
			\node [style=Empty node] (48) at (12.5, -5) {};
			\node [style=None] (49) at (17.5, -5) {};
			\node [style=None] (50) at (18, -5) {$\dots$};
			\node [style=Empty node] (51) at (0.5, -7.5) {};
			\node [style=Bead] (52) at (1.5, -7.5) {$4$};
			\node [style=Bead] (53) at (13.5, -7.5) {$2$};
			\node [style=Bead] (54) at (9.5, -7.5) {$5$};
			\node [style=Bead] (55) at (6.5, -7.5) {$1$};
			\node [style=Bead] (56) at (7.5, -7.5) {$6$};
			\node [style=Bead] (57) at (10.5, -7.5) {$3$};
			\node [style=Empty node] (58) at (2.5, -7.5) {};
			\node [style=Empty node] (59) at (3.5, -7.5) {};
			\node [style=Empty node] (60) at (5.5, -7.5) {};
			\node [style=Empty node] (61) at (4.5, -7.5) {};
			\node [style=Empty node] (62) at (11.5, -7.5) {};
			\node [style=Empty node] (63) at (12.5, -7.5) {};
			\node [style=None] (64) at (17.5, -7.5) {};
			\node [style=None] (65) at (18, -7.5) {$\dots$};
			\node [style=Empty node] (66) at (13.5, -5) {};
			\node [style=Empty node] (67) at (14.5, -5) {};
			\node [style=Empty node] (68) at (8.5, -7.5) {};
			\node [style=Empty node] (69) at (14.5, -7.5) {};
			\node [style=None] (70) at (4, -0.425) {};
			\node [style=None] (71) at (8, -2.075) {};
			\node [style=None] (72) at (5, -2.925) {};
			\node [style=None] (73) at (9, -4.575) {};
			\node [style=None] (74) at (9, -5.425) {};
			\node [style=None] (75) at (13, -7.075) {};
			\node [style=None] (76) at (-3, 0) {$(0,2,0,0,1,0)$};
			\node [style=None] (77) at (-3, -2.5) {$(0,1,0,0,1,0)$};
			\node [style=None] (78) at (-3, -5) {$(0,1,0,0,0,0)$};
			\node [style=None] (79) at (-3, -7.5) {$(0,0,0,0,0,0)$};
			\node [style=Empty node] (80) at (0.5, -11) {};
			\node [style=Bead] (81) at (1.5, -11) {$4$};
			\node [style=Bead] (82) at (3.5, -11) {$2$};
			\node [style=Bead] (83) at (4.5, -11) {$5$};
			\node [style=Bead] (84) at (6.5, -11) {$1$};
			\node [style=Bead] (85) at (7.5, -11) {$6$};
			\node [style=Bead] (86) at (10.5, -11) {$3$};
			\node [style=Empty node] (87) at (2.5, -11) {};
			\node [style=Empty node] (88) at (8.5, -11) {};
			\node [style=Empty node] (89) at (5.5, -11) {};
			\node [style=Empty node] (90) at (9.5, -11) {};
			\node [style=Empty node] (91) at (11.5, -11) {};
			\node [style=Empty node] (92) at (12.5, -11) {};
			\node [style=None] (93) at (17.5, -11) {};
			\node [style=None] (94) at (18, -11) {$\dots$};
			\node [style=Empty node] (95) at (0.5, -13.5) {};
			\node [style=Bead] (96) at (1.5, -13.5) {$4$};
			\node [style=Bead] (97) at (8.5, -13.5) {$2$};
			\node [style=Bead] (98) at (4.5, -13.5) {$5$};
			\node [style=Bead] (99) at (6.5, -13.5) {$1$};
			\node [style=Bead] (100) at (7.5, -13.5) {$6$};
			\node [style=Bead] (101) at (10.5, -13.5) {$3$};
			\node [style=Empty node] (102) at (2.5, -13.5) {};
			\node [style=Empty node] (103) at (3.5, -13.5) {};
			\node [style=Empty node] (104) at (5.5, -13.5) {};
			\node [style=Empty node] (105) at (9.5, -13.5) {};
			\node [style=Empty node] (106) at (11.5, -13.5) {};
			\node [style=Empty node] (107) at (12.5, -13.5) {};
			\node [style=None] (108) at (17.5, -13.5) {};
			\node [style=None] (109) at (18, -13.5) {$\dots$};
			\node [style=Empty node] (110) at (13.5, -11) {};
			\node [style=Empty node] (111) at (14.5, -11) {};
			\node [style=Empty node] (112) at (13.5, -13.5) {};
			\node [style=Empty node] (113) at (14.5, -13.5) {};
			\node [style=Empty node] (114) at (0.5, -16) {};
			\node [style=Bead] (115) at (1.5, -16) {$4$};
			\node [style=Bead] (116) at (13.5, -16) {$2$};
			\node [style=Bead] (117) at (4.5, -16) {$5$};
			\node [style=Bead] (118) at (6.5, -16) {$1$};
			\node [style=Bead] (119) at (7.5, -16) {$6$};
			\node [style=Bead] (120) at (10.5, -16) {$3$};
			\node [style=Empty node] (121) at (2.5, -16) {};
			\node [style=Empty node] (122) at (3.5, -16) {};
			\node [style=Empty node] (123) at (5.5, -16) {};
			\node [style=Empty node] (124) at (9.5, -16) {};
			\node [style=Empty node] (125) at (11.5, -16) {};
			\node [style=Empty node] (126) at (12.5, -16) {};
			\node [style=None] (127) at (17.5, -16) {};
			\node [style=None] (128) at (18, -16) {$\dots$};
			\node [style=Empty node] (129) at (0.5, -18.5) {};
			\node [style=Bead] (130) at (1.5, -18.5) {$4$};
			\node [style=Bead] (131) at (13.5, -18.5) {$2$};
			\node [style=Bead] (132) at (4.5, -18.5) {$5$};
			\node [style=Bead] (133) at (6.5, -18.5) {$1$};
			\node [style=Bead] (134) at (7.5, -18.5) {$6$};
			\node [style=Bead] (135) at (15.5, -18.5) {$3$};
			\node [style=Empty node] (136) at (2.5, -18.5) {};
			\node [style=Empty node] (137) at (3.5, -18.5) {};
			\node [style=Empty node] (138) at (5.5, -18.5) {};
			\node [style=Empty node] (139) at (9.5, -18.5) {};
			\node [style=Empty node] (140) at (11.5, -18.5) {};
			\node [style=Empty node] (141) at (12.5, -18.5) {};
			\node [style=None] (142) at (17.5, -18.5) {};
			\node [style=None] (143) at (18, -18.5) {$\dots$};
			\node [style=Empty node] (144) at (8.5, -16) {};
			\node [style=Empty node] (145) at (14.5, -16) {};
			\node [style=Empty node] (146) at (8.5, -18.5) {};
			\node [style=Empty node] (147) at (14.5, -18.5) {};
			\node [style=None] (148) at (4, -11.425) {};
			\node [style=None] (149) at (8, -13.075) {};
			\node [style=None] (150) at (9, -13.925) {};
			\node [style=None] (151) at (13, -15.575) {};
			\node [style=None] (152) at (11, -16.425) {};
			\node [style=None] (153) at (15, -18.075) {};
			\node [style=None] (154) at (-3, -11) {$(0,2,1,0,0,0)$};
			\node [style=None] (155) at (-3, -13.5) {$(0,1,1,0,0,0)$};
			\node [style=None] (156) at (-3, -16) {$(0,0,1,0,0,0)$};
			\node [style=None] (157) at (-3, -18.5) {$(0,0,0,0,0,0)$};
			\node [style=Empty node] (158) at (15.5, -11) {};
			\node [style=Empty node] (159) at (16.5, -11) {};
			\node [style=Empty node] (160) at (15.5, -13.5) {};
			\node [style=Empty node] (161) at (16.5, -13.5) {};
			\node [style=Empty node] (162) at (15.5, -16) {};
			\node [style=Empty node] (163) at (16.5, -16) {};
			\node [style=Empty node] (164) at (10.5, -18.5) {};
			\node [style=Empty node] (165) at (16.5, -18.5) {};
			\node [style=None] (166) at (-6, -9.25) {};
			\node [style=None] (167) at (18.5, -9.25) {};
			\node [style=Empty node] (168) at (15.5, 0) {};
			\node [style=Empty node] (169) at (16.5, 0) {};
			\node [style=Empty node] (170) at (15.5, -2.5) {};
			\node [style=Empty node] (171) at (16.5, -2.5) {};
			\node [style=Empty node] (172) at (15.5, -5) {};
			\node [style=Empty node] (173) at (16.5, -5) {};
			\node [style=Empty node] (174) at (15.5, -7.5) {};
			\node [style=Empty node] (175) at (16.5, -7.5) {};
			\end{pgfonlayer}
			\begin{pgfonlayer}{edgelayer}
			\draw (0) to (13.center);
			\draw (15) to (28.center);
			\draw (36) to (49.center);
			\draw (51) to (64.center);
			\draw [style=Move it] (70.center) to (71.center);
			\draw [style=Move it] (72.center) to (73.center);
			\draw [style=Move it] (74.center) to (75.center);
			\draw (80) to (93.center);
			\draw (95) to (108.center);
			\draw (114) to (127.center);
			\draw (129) to (142.center);
			\draw [style=Move it] (148.center) to (149.center);
			\draw [style=Move it] (150.center) to (151.center);
			\draw [style=Move it] (152.center) to (153.center);
			\draw [style=Extra box] (166.center) to (167.center);
			\end{pgfonlayer}
			\end{tikzpicture}
			\caption{For $N=6$, $r=5$, $m=3$ and $\mu=(5,3,3,2,2,1)$, the diagrams above the dashed line show all the values of $\alpha$ and $v$ in the process with the initial labelled abacus $w=(0,4,0,2,5,0,1,6,0,0,3,0,0,\dots)$ and the initial composition $\beta=(0,2,0,0,1,0)$. The shapes of these labelled abaci are the partitions $\gamma^{(i)}$ from Figure~\ref{Figure border strips}. If we use the initial composition $\beta=(0,2,1,0,0,0)$ instead, we obtain the diagrams below the dashed line. Compared to the previous diagrams, the first two $5$-moves are both with bead $2$. If we change the initial composition once more, this time to $\beta=(0,2,0,1,0,0)$, the process terminates when we reach $i=1$ as bead $4$ is not $5$-mobile.}
			\label{Figure process}
		\end{figure}
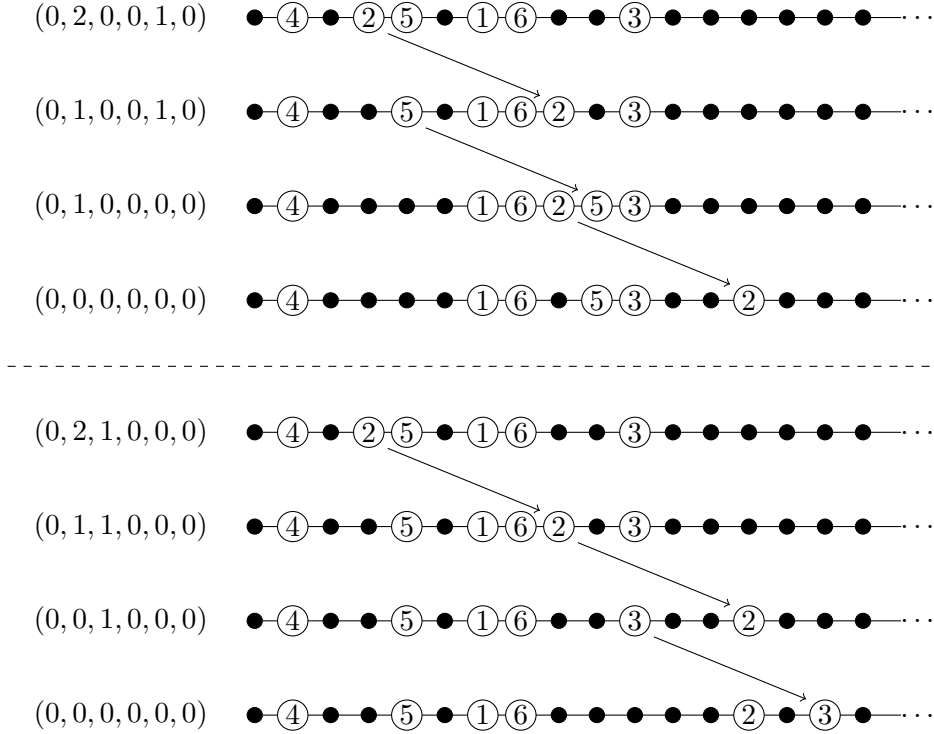
		
		The process always terminates as when we look at position $i$, the beads which are yet to be $r$-moved (that is beads $C$ such that $\alpha_C\geq 1$) lie on positions greater or equal to $i$. We write $I$ and $J$ for the set of pairs $(w,\beta)$ which are unsuccessful and successful, respectively. For $(w,\beta)\in I$, write $B$ for the label of the non-$r$-mobile bead which terminated the process and $C$ for the label of the bead that bead $B$ $r$-collided with. We define a labelled abacus $w'$ to be obtained from $w$ by swapping beads $B$ and $C$. We also define a sequence $\beta'$ of length $N$ by
		\begin{align}\label{Al beta}
		\beta'_j=\begin{cases*}
		\beta_B-\frac{w^{-1}(C) - w^{-1}(B)}{r} & if $j=B$,\\
		\beta_C+\frac{w^{-1}(C) - w^{-1}(B)}{r} & if $j=C$,\\
		\beta_j & otherwise.
		\end{cases*}
		\end{align}
		We then define $\epsilon(w,\beta)$ as $(w',\beta')$. For $(w,\beta)\in J$, we define a labelled abacus $\psi(w,\beta)$, which is the labelled abacus $v$ at the end of the process. See Figure~\ref{Figure maps from I and J} for an example.
		
		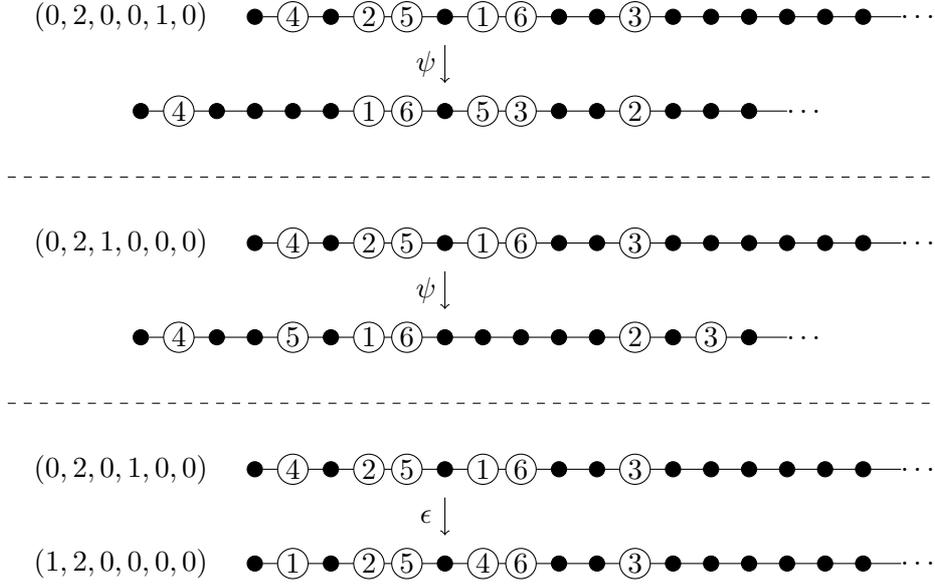
\begin{figure}[ht]
			\begin{tikzpicture}[x=0.5cm, y=0.5cm]
			\begin{pgfonlayer}{nodelayer}
			\node [style=Empty node] (0) at (4, 0) {};
			\node [style=Bead] (1) at (5, 0) {$4$};
			\node [style=Bead] (2) at (7, 0) {$2$};
			\node [style=Bead] (3) at (8, 0) {$5$};
			\node [style=Bead] (4) at (10, 0) {$1$};
			\node [style=Bead] (5) at (11, 0) {$6$};
			\node [style=Bead] (6) at (14, 0) {$3$};
			\node [style=Empty node] (7) at (6, 0) {};
			\node [style=Empty node] (8) at (12, 0) {};
			\node [style=Empty node] (9) at (9, 0) {};
			\node [style=Empty node] (10) at (13, 0) {};
			\node [style=Empty node] (11) at (15, 0) {};
			\node [style=Empty node] (12) at (16, 0) {};
			\node [style=None] (13) at (21, 0) {};
			\node [style=None] (14) at (21.5, 0) {$\dots$};
			\node [style=Empty node] (15) at (17, 0) {};
			\node [style=Empty node] (16) at (18, 0) {};
			\node [style=None] (17) at (0.5, 0) {$(0,2,0,0,1,0)$};
			\node [style=Empty node] (18) at (1, -2.5) {};
			\node [style=Bead] (19) at (2, -2.5) {$4$};
			\node [style=Bead] (20) at (14, -2.5) {$2$};
			\node [style=Bead] (21) at (10, -2.5) {$5$};
			\node [style=Bead] (22) at (7, -2.5) {$1$};
			\node [style=Bead] (23) at (8, -2.5) {$6$};
			\node [style=Bead] (24) at (11, -2.5) {$3$};
			\node [style=Empty node] (25) at (3, -2.5) {};
			\node [style=Empty node] (26) at (4, -2.5) {};
			\node [style=Empty node] (27) at (6, -2.5) {};
			\node [style=Empty node] (28) at (5, -2.5) {};
			\node [style=Empty node] (29) at (12, -2.5) {};
			\node [style=Empty node] (30) at (13, -2.5) {};
			\node [style=None] (31) at (18, -2.5) {};
			\node [style=None] (32) at (18.5, -2.5) {$\dots$};
			\node [style=Empty node] (33) at (9, -2.5) {};
			\node [style=Empty node] (34) at (15, -2.5) {};
			\node [style=None] (35) at (9, -0.75) {};
			\node [style=None] (36) at (9, -1.75) {};
			\node [style=None] (37) at (8.5, -1.25) {$\psi$};
			\node [style=None] (38) at (-2.5, -4.25) {};
			\node [style=None] (39) at (22, -4.25) {};
			\node [style=Empty node] (40) at (4, -6) {};
			\node [style=Bead] (41) at (5, -6) {$4$};
			\node [style=Bead] (42) at (7, -6) {$2$};
			\node [style=Bead] (43) at (8, -6) {$5$};
			\node [style=Bead] (44) at (10, -6) {$1$};
			\node [style=Bead] (45) at (11, -6) {$6$};
			\node [style=Bead] (46) at (14, -6) {$3$};
			\node [style=Empty node] (47) at (6, -6) {};
			\node [style=Empty node] (48) at (12, -6) {};
			\node [style=Empty node] (49) at (9, -6) {};
			\node [style=Empty node] (50) at (13, -6) {};
			\node [style=Empty node] (51) at (15, -6) {};
			\node [style=Empty node] (52) at (16, -6) {};
			\node [style=None] (53) at (21, -6) {};
			\node [style=None] (54) at (21.5, -6) {$\dots$};
			\node [style=Empty node] (55) at (17, -6) {};
			\node [style=Empty node] (56) at (18, -6) {};
			\node [style=None] (57) at (0.5, -6) {$(0,2,1,0,0,0)$};
			\node [style=Empty node] (58) at (19, -6) {};
			\node [style=Empty node] (59) at (20, -6) {};
			\node [style=None] (60) at (9, -6.75) {};
			\node [style=None] (61) at (9, -7.75) {};
			\node [style=None] (62) at (8.5, -7.25) {$\psi$};
			\node [style=Empty node] (63) at (1, -8.5) {};
			\node [style=Bead] (64) at (2, -8.5) {$4$};
			\node [style=Bead] (65) at (14, -8.5) {$2$};
			\node [style=Bead] (66) at (5, -8.5) {$5$};
			\node [style=Bead] (67) at (7, -8.5) {$1$};
			\node [style=Bead] (68) at (8, -8.5) {$6$};
			\node [style=Bead] (69) at (16, -8.5) {$3$};
			\node [style=Empty node] (70) at (3, -8.5) {};
			\node [style=Empty node] (71) at (4, -8.5) {};
			\node [style=Empty node] (72) at (6, -8.5) {};
			\node [style=Empty node] (73) at (10, -8.5) {};
			\node [style=Empty node] (74) at (12, -8.5) {};
			\node [style=Empty node] (75) at (13, -8.5) {};
			\node [style=None] (76) at (18, -8.5) {};
			\node [style=None] (77) at (18.5, -8.5) {$\dots$};
			\node [style=Empty node] (78) at (9, -8.5) {};
			\node [style=Empty node] (79) at (15, -8.5) {};
			\node [style=Empty node] (81) at (11, -8.5) {};
			\node [style=Empty node] (82) at (17, -8.5) {};
			\node [style=None] (83) at (-2.5, -10.25) {};
			\node [style=None] (84) at (22, -10.25) {};
			\node [style=Empty node] (85) at (4, -12) {};
			\node [style=Bead] (86) at (5, -12) {$4$};
			\node [style=Bead] (87) at (7, -12) {$2$};
			\node [style=Bead] (88) at (8, -12) {$5$};
			\node [style=Bead] (89) at (10, -12) {$1$};
			\node [style=Bead] (90) at (11, -12) {$6$};
			\node [style=Bead] (91) at (14, -12) {$3$};
			\node [style=Empty node] (92) at (6, -12) {};
			\node [style=Empty node] (93) at (12, -12) {};
			\node [style=Empty node] (94) at (9, -12) {};
			\node [style=Empty node] (95) at (13, -12) {};
			\node [style=Empty node] (96) at (15, -12) {};
			\node [style=Empty node] (97) at (16, -12) {};
			\node [style=None] (98) at (21, -12) {};
			\node [style=None] (99) at (21.5, -12) {$\dots$};
			\node [style=Empty node] (100) at (17, -12) {};
			\node [style=Empty node] (101) at (18, -12) {};
			\node [style=None] (102) at (0.5, -12) {$(0,2,0,1,0,0)$};
			\node [style=None] (103) at (9, -12.75) {};
			\node [style=None] (104) at (9, -13.75) {};
			\node [style=None] (105) at (8.5, -13.25) {$\epsilon$};
			\node [style=Empty node] (106) at (4, -14.5) {};
			\node [style=Bead] (107) at (10, -14.5) {$4$};
			\node [style=Bead] (108) at (7, -14.5) {$2$};
			\node [style=Bead] (109) at (8, -14.5) {$5$};
			\node [style=Bead] (110) at (5, -14.5) {$1$};
			\node [style=Bead] (111) at (11, -14.5) {$6$};
			\node [style=Bead] (112) at (14, -14.5) {$3$};
			\node [style=Empty node] (113) at (6, -14.5) {};
			\node [style=Empty node] (114) at (12, -14.5) {};
			\node [style=Empty node] (115) at (9, -14.5) {};
			\node [style=Empty node] (116) at (13, -14.5) {};
			\node [style=Empty node] (117) at (15, -14.5) {};
			\node [style=Empty node] (118) at (16, -14.5) {};
			\node [style=None] (119) at (21, -14.5) {};
			\node [style=None] (120) at (21.5, -14.5) {$\dots$};
			\node [style=Empty node] (121) at (17, -14.5) {};
			\node [style=Empty node] (122) at (18, -14.5) {};
			\node [style=None] (123) at (0.5, -14.5) {$(1,2,0,0,0,0)$};
			\node [style=Empty node] (124) at (19, 0) {};
			\node [style=Empty node] (125) at (20, 0) {};
			\node [style=Empty node] (126) at (16, -2.5) {};
			\node [style=Empty node] (127) at (17, -2.5) {};
			\node [style=Empty node] (128) at (19, -12) {};
			\node [style=Empty node] (129) at (20, -12) {};
			\node [style=Empty node] (130) at (20, -14.5) {};
			\node [style=Empty node] (131) at (19, -14.5) {};
			\end{pgfonlayer}
			\begin{pgfonlayer}{edgelayer}
			\draw (0) to (13.center);
			\draw (18) to (31.center);
			\draw [style=Move it] (35.center) to (36.center);
			\draw [style=Extra box] (38.center) to (39.center);
			\draw (40) to (53.center);
			\draw [style=Move it] (60.center) to (61.center);
			\draw (63) to (76.center);
			\draw [style=Extra box] (83.center) to (84.center);
			\draw (85) to (98.center);
			\draw [style=Move it] (103.center) to (104.center);
			\draw (106) to (119.center);
			\end{pgfonlayer}
			\end{tikzpicture}
			\caption{Let $w$ be the labelled abacus from Figure~\ref{Figure process}. As observed, we have $(w,(0,2,0,0,1,0)), (w,(0,2,1,0,0,0))\in J$ and $(w,(0,2,0,1,0,0))\in I$. The diagrams above display the images of maps $\epsilon$ and $\psi$ applied to these three pairs.}
			\label{Figure maps from I and J}
		\end{figure} 
		
		We claim that (\ref{Eq anti pMN rule}) follows once we establish the following two statements:
		\begin{enumerate}[label=\textnormal{(\roman*)}]
			\item The map $\epsilon$ is a weight-preserving involution on $I$, which reverses the sign.
			\item The map $\psi$ is a weight-preserving bijection from $J$ to $\bigcup_{\lambda}\Abc_N(\lambda)$, where the union is taken over partitions $\lambda\in\Par(|\mu|+rm)$ such that $\lambda/\mu$ is an $r$-decomposable skew partition. Moreover, for any $(w,\beta)\in J$ we have $\sgn(\psi(w,\beta)) = \sgn_r(\lambda/\mu)\sgn(w,\beta)$, where $\lambda$ is the shape of $\psi(w,\beta)$.
		\end{enumerate}
		We now prove this claim. We can split the final sum of (\ref{Eq LHS expansion}), which equals the left-hand side of (\ref{Eq anti pMN rule}), as
		\begin{equation*}
		\sum_{(w,\beta)\in I}\sgn(w,\beta)\wt_r(w,\beta) + \sum_{(w,\beta)\in J}\sgn(w,\beta)\wt_r(w,\beta).
		\end{equation*}
		The first sum is zero from (i), while the second sum can be rewritten, using (ii), as
		\begin{align*}
		&\sum_{\substack{\mu\subseteq\lambda\in\Par(|\mu|+rm) \\ \lambda/\mu \textnormal{ is $r$-decomposable}}}\sgn_r(\lambda/\mu)\sum_{w\in\Abc_N(\lambda)}\sgn(w)\wt(w)=\\ &\sum_{\mu\subseteq\lambda\in\Par(|\mu|+rm)}\sgn_r(\lambda/\mu)\sum_{w\in\Abc_N(\lambda)}\sgn(w)\wt(w),
		\end{align*}
		which is the right-hand side of (\ref{Eq anti pMN rule}) by (\ref{Eq antisymmetric}), as required.
		
		Thus it remains to show (i) and (ii). For (i), let $(w,\beta)\in I$ and write $(w',\beta')$ for $\epsilon(w,\beta)$. We keep the notations $B$ and $C$ for the labels of the beads such that the process terminated with non-$r$-mobile bead $B$ which $r$-collided with bead $C$. Clearly, $w'$ lies in $\Abc_N(\mu)$. We now check that $\beta$ lies in $\Com_N(m)$. Since the beads only $r$-move, we have $r\mid w^{-1}(C)-w^{-1}(B)$; thus $\beta'$ is an integral sequence. We also see that bead $C$ has not $r$-moved during the process; hence $w^{-1}(C) > w^{-1}(B)$ and we must have checked whether bead $B$ is $r$-mobile at least $\left( \left( w^{-1}(C)-w^{-1}(B)\right) /r\right) $-times. The latter statement implies that $\beta_B\geq \left( w^{-1}(C)-w^{-1}(B)\right) /r$, and in turn the entries of $\beta'$ are non-negative. From (\ref{Al beta}), clearly, $\beta'$ and $\beta$ have the same sum of entries, and hence $\beta'\in\Com_N(m)$.
		
		By Lemma~\ref{Lemma sign change}(i), we have $\sgn(w')=-\sgn(w)$. The equality of weights $\wt_r(w,\beta)=\wt_r(w',\beta')$ holds true as $w^{-1}(B)+r\beta_B=w^{-1}(C) + r(\beta_B -(w^{-1}(C)-w^{-1}(B))/r)$ and $w^{-1}(C)+r\beta_C=w^{-1}(B) + r(\beta_C +(w^{-1}(C)-w^{-1}(B))/r)$. Hence it remains to verify that $(w',\beta')$ lies in $I$ and $\epsilon(w',\beta')=(w,\beta)$.
		
		The process with $(w',\beta')$ coincides with the process with $(w,\beta)$ where $r$-moves of bead $B$ are replaced with $r$-moves of bead $C$ as long as $\beta'_C\geq (w'^{-1}(B)-w'^{-1}(C))/r$. This inequality holds true as the right-hand side is $(w^{-1}(C)-w^{-1}(B))/r$ which is at most $\beta'_C$ by (\ref{Al beta}). Hence $(w',\beta')\in I$ and the process ends with bead $C$ $r$-colliding with bead $B$. Writing $(w'',\beta'')=\epsilon(w',\beta')$, we see that $w''=w$ and since $\wt_r(w,\beta)=\wt_r(w',\beta')=\wt_r(w'',\beta'')$, we immediately conclude that also $\beta''=\beta$.
		
		We now move to (ii). Clearly, during the process, the weight of $(v,\alpha)$ does not change. Hence $\wt_r(w,\beta)= \wt_r(\psi(w,\beta),(0,0,\dots,0))=\wt(\psi(w,\beta))$, that is $\psi$ preserves weights. The rest of the claim follows from Corollary~\ref{Cor r-decomposable}. In more detail, the statement about sizes in the `moreover' part of Corollary~\ref{Cor r-decomposable} together with implications (ii)$\implies$(i) and (i)$\implies$(iii) shows that the map $\psi$ takes values in the desired set and is surjective, respectively. The uniqueness statement in the `moreover' part shows that $\psi$ is injective and the statement about signs in the `moreover' part shows that $\sgn(\psi(w,\beta)) = \sgn_r(\lambda/\mu)\sgn(w,\beta)$, where $\lambda=\sh(\psi(w,\beta))$.
	\end{proof}
	
	\begin{remark}\label{Remark special cases}
		If we let $r=1$, respectively, $m=1$ in the proof, we obtain the proof of Young's rule, respectively, the Murnaghan--Nakayama rule from \cite{LoehrAbacus10}.
	\end{remark}
	
	\subsection*{Acknowledgements}
	The author would like to thank Mark Wildon for suggesting this project and for his encouragement, as well as the anonymous referees for their comments, which helped greatly improve the manuscript.
	
	\printbibliography
\end{document}